\documentclass[12pt]{article}
\usepackage{amsfonts}
\usepackage{mathrsfs}

\usepackage[latin1]{inputenc}
\usepackage{amsmath,amssymb}
\usepackage{latexsym}
\usepackage{epstopdf}
\usepackage{geometry}   
\usepackage[active]{srcltx}
\usepackage{graphicx}
\usepackage{epstopdf}
\usepackage[
bookmarks=true,         
bookmarksnumbered=true, 
colorlinks=true, pdfstartview=FitV, linkcolor=blue, citecolor=blue,
urlcolor=blue]{hyperref}

\newtheorem {theorem}{Theorem}[section]
\newtheorem {proposition}{Proposition}[section]

\newenvironment{proof}[1][Proof]{\textbf{#1.} }{\
\rule{0.5em}{0.5em}}

\begin{document}
%
%
%
\begin{center}
{\LARGE Central limit theorem for the variable bandwidth kernel density estimators}

\bigskip


\bigskip Janet Nakarmi$^{a}$ and Hailin Sang$^{b}$$^{1}$\footnotetext[1]{Corresponding author}

\bigskip$^{a}$ Department of Mathematics, University of Central Arkansas, Conway, AR 72035, USA. E-mail address: janetn@uca.edu

\bigskip$^{b}$ Department of Mathematics, The University of Mississippi,
University, MS 38677, USA. E-mail address: sang@olemiss.edu
\end{center}

\begin{center}
\bigskip\textbf{Abstract}
\end{center}

In this paper we study the ideal variable bandwidth kernel density estimator introduced by McKay \cite{McKay a, McKay b} and Jones et al. \cite{JonesMcKayHu}  and the plug-in practical version of the variable bandwidth kernel estimator with two sequences of bandwidths as in Gin\a'{e} and Sang \cite{GineSang1}.  Based on the bias and variance analysis of the ideal and true variable bandwidth kernel density estimators, we study the central limit theorems for each of them.  \\



\section{Introduction}

Suppose that $X_i,i \in \mathbb{N}$, are independent identically distributed (i.i.d.) observations with density function $f(t)$,  $t \in \mathbb{R}^d$. Let $K$ to be a symmetric probability kernel satisfying some differentiability properties. The classical kernel density estimator 
\begin{equation}\label{classical}
\hat{f}(t;h_{n}) = \frac{1}{nh_{n}^d}\sum_{i=1}^n K\left(\frac{t-X_i}{h_{n}}\right), 
\end{equation}
where $h_n$ is the bandwidth sequence with $h_n\rightarrow 0, nh_n^d\rightarrow \infty$, and its properties have been well studied in the literature. The variance of (\ref{classical}) has order $O((nh_n^d)^{-1})$ and the bias has order $O(h_n^{2})$ if $f(t)$ has bounded second order partial derivatives. See Silverman \cite{Silverman} and Wand and Jones \cite{WJ} for the literature on kernel density estimation. For $k=(k_1,\dots,k_d)\in (\mathbb N\cup\{0\})^d$,  set $|k|=\sum_{i=1}^dk_i$, one may obtain bias with order $O(h_n^{4})$ for the estimator (\ref{classical}) if the fourth order kernel function $K(x)$ is allowed: $\int_{\mathbb{R}^d} K(x)dx=1$ and $\int_{\mathbb{R}^d} x_1^{k_1}\cdots x_d^{k_d} K(x)dx=0$ for $|k|=1,2,3$. Nevertheless,  $\hat{f}(t;h_{n})$ in  (\ref{classical}) may take negative values and therefore not a true density function in this case since $K(x)$ may take negative values. For example, see Marron \cite{Marron}. In this paper we study the following multidimensional version of the variable bandwidth kernel density estimator proposed by McKay \cite{McKay a, McKay b}:
 \begin{equation}\label{ideal}
 \bar{f}(t;h_n)=\frac{1}{nh^d_n} \sum_{i=1}^{n}\alpha^d(f(X_i))K(h_n^{-1}\alpha(f(X_i))(t-X_i)),
 \end{equation}
where $\alpha(s)$ is a smooth function of the form
\begin{equation}\label{alp}
\alpha(s):=cp^{1/2}(s/c^2).
\end{equation}
The function $p$ has at least fourth order derivative and satisfies $p(x)\geq 1$ for all $x$ and $p(x)=x$ for all $x\geq t_0$ for some $1 \le t_0<\infty$, and a fixed number $c$, where $0<c<\infty$. The equation (\ref{ideal}) is a variable bandwidth kernel density estimator since the bandwidth has form $h_n/\alpha(f(X_i))$ if we rewrite (\ref{ideal}) in the form of  the classical one, (\ref{classical}). 

The study of variable bandwidth kernel density estimation goes back to Abramson \cite{Abramson}. Abramson proposed the following estimator 
\begin{equation}\label{abramson}
f_A(t;h_n)=\frac{1}{nh^d_n}\sum_{i=1}^n \gamma^d(t, X_i)K(h_n^{-1}\gamma (t, X_i)(t-X_i)),
\end{equation}
where $\gamma (t,s)=(f(s)\vee f(t)/10)^{1/2}$.  The bandwidth $ h_n/\gamma (t, X_i)$ at each observation $X_i$ is inversely proportional to $f^{1/2}(X_i)$ if $f(X_i)\ge f(t)/10$.  Notice that (\ref{ideal}) also has the square root law since $\alpha(f(X_i))=f^{1/2}(X_i)$ if $f(X_i)\ge t_0c^2$ by the definition of the function $p(x)$. The estimator (\ref{ideal}) or (\ref{abramson}) has clipping procedure in (\ref{alp}) or $\gamma (t,s)$ since they make the true bandwidth $ h_n/\alpha(f(X_i))\ge h_n/c$ or $ h_n/\gamma (t, X_i)\ge 10^{1/2} h_n/f(t)^{1/2}$. The clipping procedures prevent too much contribution to the density estimation at $t$ if the observation $X_i$ is too far away from $t$. Abramson showed that this square root law and the clipping procedure improve the bias from the order of $h_n^2$ to the order of $h_n^4$ for the estimator (\ref{abramson}) while at the same time keep the variance at the order of $(nh_n^d)^{-1}$ if $f(t)\ne 0$ and $f(x)$ has fourth order continuous derivatives at $t$. So, one has a {\it non-negative} estimator of the density that performs asymptotically as a kernel estimator based on a fourth order (hence, partly negative) kernel. However,  this variable bandwidth estimator (\ref{abramson}) is not a density function of a true probability measure since the integral of $f_A(t;h_n)$ over $t$ is not  $1$.  

Terrell and Scott \cite{TerrellScott} and  McKay \cite{McKay b} showed that the following modification of the Abramson estimator without the `clipping filter' $(f(t)/10)^{1/2}$ on $f^{1/2}(X_i)$ studied in Hall and Marron \cite{HallMarron},
\begin{equation}\label{HM}
f_{HM}(t;h_n)=\frac{1}{nh^d_n}\sum_{i=1}^n f^{d/2}(X_i)K(h_n^{-1}f^{1/2}(X_i)(t-X_i)),
\end{equation}
which has integral $1$ and thus is a true probability density, may have bias of order much larger than $h_n^4$.  Therefore, the clipping is necessary for such bias reduction.   In the case $d=1$, Hall, Hu and Marron \cite{HallHuMarron} proposed the estimator
\begin{equation}\label{HHM}
f_{HHM}(t;h_n)=\frac{1}{n h_n}\sum_{i=1}^{n}K\left(\frac{t-X_i}{h_n} f^{1/2}(X_i)\right) f^{1/2}(X_i)I(|t-X_{i}|<h_nB)
\end{equation}
where $B$ is a fixed constant; see also Novak \cite{Novak} for a similar estimator.
This estimator is non-negative and achieves the desired bias reduction but, like Abramson's, it does not integrate to 1.

In conclusion, it seems that the estimator (\ref{ideal}) has all the advantages: it is a true density function with square root law and smooth clipping procedure. However, notice that this estimator and all the other variable bandwidth kernel density estimators are not applicable in practice since they all include the studied density function $f$. Therefore, we call them ideal estimators in the literature.  Hall and Marron \cite{HallMarron} studied a true density estimator 
\begin{equation*} 
\hat{f}_{HM}(t; h_{1,n}, h_{2,n})=\frac{1}{nh_{2,n}^d}\sum_{i=1}^{n}K\left(\frac{t-X_i}{h_{2,n}}\hat{f}^{1/2}(X_i;h_{1,n})\right)\hat{f}^{d/2}(X_i;h_{1,n}),
\end{equation*}
by plugging in a pilot estimator, the classical estimator (\ref{classical}), into the estimator (\ref{HM}). 
Here the bandwidth sequence $h_{2,n}$ is the $h_n$ as in (\ref{HM}) and the bandwidth sequence $h_{1,n}$ is applied in the classical kernel density estimator (\ref{classical}), i.e., 
\begin{equation*}
\hat{f}(t;h_{1,n}) = \frac{1}{nh_{1,n}^d}\sum_{i=1}^n K\left(\frac{t-X_i}{h_{1,n}}\right).
\end{equation*}
They took the Taylor expansion of $K\left(\frac{t-X_i}{h_{2,n}}\hat{f}^{1/2}(X_i;h_{1,n})\right)$ at $K\left(\frac{t-X_i}{h_{2,n}}f^{1/2}(X_i)\right)$ and then proved that the discrepancy between the true estimator $\hat{f}_{HM}(t; h_{1,n}, h_{2,n})$ and the ideal version (\ref{HM}) has asymptotic convergence rate $O_P(n^{-4/(8+d)})$ pointwise. 
By applying this Taylor decomposition, McKay \cite{McKay b} studied convergence of plug-in true estimator of (\ref{ideal}) in probability and pointwise. Gin\'{e} and Sang \cite{GineSang, GineSang1} studied plug-in true estimators of (\ref{HHM}) and (\ref{ideal}) for one and d-dimensional observations. They proved that the discrepancy between the true estimator and the true value converges uniformly over a data adaptive region at a rate of $O_{a.s.}((\log n/n)^{4/(8+d)})$ by applying empirical process techniques.  The true estimator in  Gin\'{e} and Sang \cite{GineSang1} has the following form
\begin{equation}\label{true}
\hat{f}(t; h_{1,n}, h_{2,n})=\frac{1}{nh_{2,n}^d}\sum_{i=1}^{n}K\left(\frac{t-X_i}{h_{2,n}}\alpha(\hat{f}(X_i;h_{1,n}))\right)\alpha^d(\hat{f}(X_i;h_{1,n})).
\end{equation}
In this paper, we concentrate on the study of central limit theorem of the true estimator (\ref{true}). 

The paper has the following structure. Section \ref{preliminary} introduces the decompositions which will be applied throughout the paper. Section \ref{BIAS} gives the exact bias formula. In Section \ref{variance}, we obtain an exact formula for the variance of the ideal estimator. Based on the study in Sections \ref{BIAS} and \ref{variance}, we provide a central limit theorem for the true estimator in Section \ref{central}.  The simulation study in Section \ref{sim} demonstrates the advantage of the variable bandwidth kernel estimation. 
\section{Preliminary decomposition}\label{preliminary}
For convenience, we adopt the notations as in Gin\'{e} and Sang \cite{GineSang1} for the Taylor series expansion of $K\left(\frac{t-X_i}{h_{2,n}}\alpha(\hat f(X_i;h_{1,n}))\right)$ at $K\left(\frac{t-X_i}{h_{2,n}}\alpha(f(X_i))\right)$.  We also give statements without detailed explanation. For details, readers are referred to Gin\'{e} and Sang \cite{GineSang1}. ${\cal P}_{C,k}$ will denote the set of densities on $\mathbb R^d$ for which they and their partial derivatives of  order $k$ or lower are bounded by $C<\infty$ and are uniformly continuous. We say that a function $g$ is in $C^l(\Omega)$ if it and its first $l$ derivatives are bounded and uniformly continuous on $\Omega$.
 
 Define
$\delta(t)=\delta(t,n)$ by the equation
\begin{equation*} 
\delta(t)=\frac{\alpha(\hat f(t;h_{1,n}))-\alpha (f(t))}{\alpha(f(t))}.
\end{equation*}
Then, 
\begin{equation}\label{alphahat}
\alpha(\hat f(t;h_{1,n}))=\alpha(f(t))(1+\delta(t))
\end{equation}
and
\begin{equation}\label{deltabd}
|\delta(t)|\le B c^{-2}|\hat f(t;h_{1,n})-f(t)|
\end{equation}
for a constant $B$ that depends only on the function $p$. Here the constant $c$ and the function $p$ are applied in the definition of $\alpha(\cdot)$ in (\ref{alp}).   
Although we study the asymptotics of the true estimator pointwise, the uniform asymptotic behavior of the quantity $\delta(\cdot)$ is needed in the latter analysis. Define
$$D(t;h_{1,n})=\hat f(t;h_{1,n})-\mathbb{E}\hat f(t;h_{1,n})\ \ {\rm and}\ \ b(t;h_{1,n})=\mathbb{E}\hat f(t;h_{1,n})-f(t).$$
Note that for $f\in{\cal P}_{C,2}$,
$\sup_{t\in \mathbb{R}^d} |b(t;h_{1,n})|=O(h_{1,n}^2)$, 
and by Gin\'e and Guillou \cite{GineGuillou2}, 
\begin{equation*} 
\sup_{t\in \mathbb{R}^d}|D(t;h_{1,n})|=O_{a.s.}\left(\sqrt{\frac{\log h_{1,n}^{-1}}{nh_{1,n}^d}}\right)
\end{equation*}
 for $f\in{\cal P}_{C,0}$.
Denote 
\begin{equation}\label{u}
\sqrt{\frac{\log h_{1,n}^{-1}}{nh_{1,n}^d}}+h_{1,n}^2:=U(h_{1,n}).
\end{equation}
Then we have, 
\begin{equation*} 
\sup_{t\in \mathbb{R}^d}|\hat f(t;h_{1,n})-f(t)|=\sup_{t\in \mathbb{R}^d}|D(t;h_{1,n})+b(t;h_{1,n})|=O_{a.s.}\left(U(h_{1,n})\right)
\end{equation*}
and 
\begin{equation}\label{zero}
\sup_{t\in \mathbb{R}^d}|\delta(t)|=O_{a.s.}\left(U(h_{1,n})\right)
\end{equation}
for $f\in{\cal P}_{C,2}$. By the definition of $\delta(t)$, we also have, 
\begin{eqnarray}\label{deltatalor}
\delta(t)=\frac{\alpha'(f(t))[\hat f(t;h_{1,n})-f(t)]}{\alpha(f(t))}
+\frac{\alpha''(\eta)[\hat f(t;h_{1,n})-f(t)]^2}{2\alpha(f(t))}
\end{eqnarray}
where $\eta=\eta(t,h_{1,n})\ge 0$ is between $\hat f(t;h_{1,n})$ and $f(t)$. Notice that $|\alpha''(\eta( t,h_{1,n}))|\le c^{-3}A$ for some constant $A$ which depends only on the clipping function $p$.
It is also convenient to record the following expansion of $\alpha^d(\hat f)$ implied by (\ref{alphahat}) and (\ref{zero}):
\begin{equation} \label{delta1}
\alpha^d(\hat f(t;h_{1,n}))=\alpha^d(f(t))(1+d\delta(t))+\delta_1(t)
\end{equation}
with\begin{equation*} 
\|\delta_1\|_\infty=O_{\rm a.s.}(\|\delta\|_\infty^2)\ \ {\rm for\ }\ f\in{\cal P}_{C,2}.
\end{equation*}
Hence, by (\ref{deltabd}) and (\ref{zero}), 
\begin{equation*} 
\|\delta_1\|_\infty=O_{\rm a.s.}(\|\hat f_n(\cdot;h_{1,n})-f(\cdot)\|_\infty^2)\ \ {\rm for\ }\  f\in{\cal P}_{C,2}.
\end{equation*}
Set 
\begin{equation} \label{L1Lfunctions}
L_1(t)=\sum_{i=1}^dt_iK'_i(t)\ \ {\rm and}\ \  L(t)=dK(t)+L_1(t),\ \ t\in{\mathbb R}^d,
\end{equation}
where $K'_i$ denotes the partial derivative of $K$ in the direction of the $i$-th coordinate, and $t_i$ denotes the $i$-th coordinate of $t\in\mathbb R^d$.  By symmetry and integration by parts, we notice that $L$ is a second order kernel.

We then have the following Taylor series expansion
\begin{eqnarray} \label{expK}
&&K\left(\frac{t-X_i}{h_{2,n}}\alpha(\hat f(X_i;h_{1,n}))\right)=K\left(\frac{t-X_i}{h_{2,n}}\alpha(f(X_i))\right)
\notag\\
&&\!\!+\sum_{j=1}^dK'_j\left(\frac{t-X_i}{h_{2,n}}\alpha(f(X_i))\right)\frac{(t-X_i)_j}{h_{2,n}}\alpha(f(X_i))\delta(X_i)+\delta_2(t;X_i),
\end{eqnarray}
where
\begin{equation*} 
\delta_2(t,X_i)=\sum_{j,\ell=1}^dK''_{j,\ell}(\xi)\frac{(t-X_i)_j(t-X_i)_\ell}{2h_{2,n}^2}\alpha^2(f(X_i))\delta^2(X_i),
\end{equation*}
$\xi$  being a (random) number between $\frac{t-X_i}{h_{2,n}}\alpha(f(X_i))$ and
$\frac{t-X_i}{h_{2,n}}\alpha(f(X_i))+
\frac{t-X_i}{h_{2,n}}\alpha(f(X_i))\delta (X_i)$.  By the analysis in Gin\'{e} and Sang \cite{GineSang1}, 
\begin{equation}\label{delta2bd'}
\sup_{t,x\in\mathbb R^d}|\delta_2(t,x)|=O_{\rm a.s.}\left(\|\hat f(\cdot;h_{1,n})-f(\cdot)\|_\infty^2\right)=O_{a.s.}(U^2(h_{1,n}))
\end{equation}
if $f\in{\cal P}_{C,2}$.
Therefore using equation $(\ref{L1Lfunctions})$, Taylor series expansion of $K$ in $(\ref{expK})$, and expansion of $\alpha^d$ in $(\ref{delta1})$, we have
\begin{eqnarray}
&&\hat f(t;h_{1,n},h_{2,n})=\bar f(t;h_{2,n})\notag\\
&&\!\!+\frac{1}{nh_{2,n}^d}\sum_{i=1}^{n}L\left(\frac{t-X_i}{h_{2,n}}\alpha(f(X_i))\right)
\alpha^d(f(X_i))\delta (X_i)\label{e'}\\
&&\!\!+\frac{1}{nh_{2,n}^d}\sum_{i=1}^{n}\Bigg[K\left(\frac{t-X_i}{h_{2,n}}\alpha(f(X_i))\right)\delta_1(X_i)+\alpha^d(f(X_i))\delta_2(t,X_i)\notag\\
&&~~~~~~~~~~~~~~~~~~~~~+dL_1\left(\frac{t-X_i}{h_{2,n}}\alpha(f(X_i))\right)\alpha^d(f(X_i))\delta^2(X_i)\Bigg]\label{e1}\\
&&\!\!+\frac{1}{nh_{2,n}^d}\sum_{i=1}^{n}\left[L_1\left(\frac{t-X_i}{h_{2,n}}\alpha(f(X_i))\right)\delta(X_i)\delta_1(X_i)+d\alpha^d(f(X_i))\delta(X_i)\delta_2(t,X_i)\right]\notag\\
&&\!\!\label{e2}\\
&&\!\!+\frac{1}{nh_{2,n}^d}\sum_{i=1}^n \delta_2(t,X_i)\delta_1(X_i).\label{e3}
\end{eqnarray}

\section{Bias}\label{BIAS}
The following notations are necessary for the rest of the paper: for $v=(v_1,\dots,v_d)\in (\mathbb N\cup\{0\})^d$ and vector $u=(u_1,\dots, u_d)^T$,  set 
\begin{align}\label{notation}
&|v|=\sum_{i=1}^dv_i, \;\; v!=v_1!\cdots v_d!, \notag\\
&D_v=D_{u_1}^{v_1}\circ\dots\circ D_{u_d}^{v_d},\;\; u^v=u_1^{v_1}\cdots u_d^{v_d},\\
&\tau_v=\int_{\mathbb R^d} u^vK(u)du,\;\;  \mu_v=\int_{\mathbb R^d} u^vK^2(u)du,\notag
\end{align}
where $D_{v}$ means that we take  $v_1$ partial order derivatives on the first coordinate,  $v_2$ partial order derivatives on the second coordinate, until we take  $v_d$ partial order derivatives on the $d$-th coordinate.

We also define
\begin{equation*}
{\cal D}_r:=\{t\in\mathbb R^d: f(t)>r>t_0 c^2, \|t\|<1/r\},\ \ r>0.
\end{equation*}
Here, $c$ and $t_0$ are the constants that appear in the definition of the clipping function $\alpha$ in (\ref{alp}).
\begin{proposition} \label{bias12}
Let $f$ be a  density function in ${\cal P}_{C,4}$, let $p$ be a  clipping function in $C^5(\mathbb R)$,  set $\alpha(f(t))=cp^{1/2}(c^{-2}f(t))$ for some $c>0$, and define $\hat f(t;h_{1,n},  h_{2,n})$ as in (\ref{true}). Suppose that the kernel $K$ on $\mathbb{R}^d$ has the form $K(t)=\Phi(\|t\|^2)$  for some real function $\Phi$  with uniformly bounded second order derivative and with support contained in $[0,T]$, $T<\infty$. $K$ is non-negative and  integrates to 1. For the quantity $U(h_{1,n})$ defined in (\ref{u}), assume that  $U(h_{1,n})=o(h_{2,n}^2)$. Then as $h_{2,n}\rightarrow 0$, for $t\in {\cal D}_r$, 
\begin{eqnarray*}
\mathbb{E}(\hat f(t;h_{1,n},  h_{2,n}))-f(t)
=\left(\sum_{|v|=4}\tau_vD_v(1/f)/v!\right)h_{2,n}^4 +o(h_{2,n}^4).
\end{eqnarray*}

\end{proposition}
\begin{proof}
\indent By McKay \cite{McKay a,McKay b} or Corollary 1 of Gin\a'{e} and Sang \cite{GineSang1}, the ideal estimator (\ref{ideal}) satisfies 
\begin{eqnarray} \label{mckaybi}
\mathbb{E}\bar f(t;h_{2,n})=f(t)+\left(\sum_{|v|=4}\tau_vD_v(1/f)/v!\right)h_{2,n}^4 +o(h_{2,n}^4).
\end{eqnarray}

Hence by the expansion of $\hat f(t;h_{1,n},  h_{2,n})$ in (\ref{e'})-(\ref{e3}) and equation (\ref{mckaybi}), the expectation of $\hat f(t;h_{1,n},  h_{2,n})$ is
\begin{align}
\mathbb{E}&\hat f(t;h_{1,n},h_{2,n})=f(t)+\left(\sum_{|v|=4}\tau_vD_v(1/f)/v!\right)h_{2,n}^4 +o(h_{2,n}^4)\label{idealmean}\\
&+\frac{1}{h_{2,n}^d}\mathbb{E}\left(L\left(\frac{t-X_1}{h_{2,n}}\alpha(f(X_1))\right)
\alpha^d(f(X_1))\delta (X_1)\right)\label{e8}\\
&+\mathbb{E}\Bigg[\frac{1}{nh_{2,n}^d}\sum_{i=1}^{n}\Bigg(K\left(\frac{t-X_i}{h_{2,n}}\alpha(f(X_i))\right)\delta_1(X_i)+\alpha^d(f(X_i))\delta_2(t,X_i)\notag\\
&+dL_1\left(\frac{t-X_i}{h_{2,n}}\alpha(f(X_i))\right)\alpha^d(f(X_i))\delta^2(X_i)\Bigg)\Bigg]\label{e10}\\
&+\mathbb{E}\left[\frac{1}{nh_{2,n}^d}\sum_{i=1}^{n}\left(L_1\left(\frac{t-X_i}{h_{2,n}}\alpha(f(X_i))\right)\delta(X_i)\delta_1(X_i)+d\alpha^d(f(X_i))\delta(X_i)\delta_2(t,X_i)\right)\right]\label{e11}\\
&+\mathbb{E}\left[\frac{1}{nh_{2,n}^d}\sum_{i=1}^n\delta_2(t,X_i)\delta_1(X_i)\right]\label{e12}.
\end{align}
By (3.20) of  Gin\a'{e} and Sang \cite{GineSang1} and the boundedness of $\alpha$, $K$ and $L_1$, we have 
\begin{align}
|(\ref{e10})| = O\left(U^2(h_{1,n})\right),\label{e10'}
\end{align}
\begin{align}
|(\ref{e11})|= O\left(U^3(h_{1,n})\right),\label{e11'}
\end{align}
and
\begin{align}
|(\ref{e12})|= O\left(U^4(h_{1,n})\right),\label{e12'}
\end{align}
where the $U(h_{1,n})$ is defined in (\ref{u}). We further decompose (\ref{e8}) using the decomposition  (\ref{deltatalor}) of $\delta(t)$ first, and then the decomposition of $\hat f-f$ into the random part $D$ and the bias $b$:
\begin{align}
&(\ref{e8})=\frac{1}{dh^d_{2,n}}\mathbb{E}\bigg[L\left(\frac{t-X_1}{h_{2,n}}\alpha(f(X_1))\right)
(\alpha^d)'( f(X_1)) D(X_1;h_{1,n})\bigg]\label{eps1}\\
&+\frac{1}{dh^d_{2,n}}\mathbb{E}\bigg[L\left(\frac{t-X_1}{h_{2,n}}\alpha(f(X_1))\right)
(\alpha^d)'(f(X_1))  b(X_1;h_{1,n})\label{eps2}\bigg]\\
&+\frac{1}{2h^d_{2,n}}\mathbb{E}\left[\bigg(L\left(\frac{t-X_1}{h_{2,n}}\alpha(f(X_1))\right)(\alpha^{d-1})(f(X_1))\alpha''(\eta(X_1))[\hat f(X_1;h_{1,n})-f(X_1)]^2\bigg)\right].\label{e5}
\end{align}
By (\ref{delta2bd'}) or  (3.24) of  Gin\a'{e} and Sang \cite{GineSang1}  and the boundedness of $\alpha''(\eta)$ and $L$, we obtain,
\begin{eqnarray}\label{easy}
|(\ref{e5})|=O\left(U^2(h_{1,n})\right).
\end{eqnarray}
By (3.26) of Gin\a'{e} and Sang \cite{GineSang1}, we have 
\begin{equation}\label{biasbias}
|(\ref{eps2})|=O(h_{1,n}^2h_{2,n}^2) \ \ {\rm for} \ f\in{\cal P}_{C,4}.
\end{equation}

In the following we give the estimation of (\ref{eps1}) to finish the proof of the proposition. 

Let $H$ be an integrable function of two i.i.d. random variables $X$ and $Y$. Then the $U$-statistic is
$$U_n(H)=\frac{1}{n(n-1)}\sum_{1\le i\ne j\le n}H(X_i,X_j),$$
where the variables $X_i$ are i.i.d. copies of $X$.  The second order Hoeffding projection of $H(X,Y)$ is
$\pi_2(H)(X,Y)=H(X,Y)-\mathbb{E}_XH(X,Y)-\mathbb{E}_YH(X,Y)+\mathbb{E}H.$
If we set
\begin{equation*} 
H_t(X,Y):=L\left(\frac{t-X}{h_{2,n}}\alpha(f(X))\right)
(\alpha^d)'(f(X))K\left(\frac{X-Y}{h_{1,n}}\right),
\end{equation*}
then we can decompose the following quantity into a diagonal term and a $U$-statistic term,
\begin{align}
&\frac{1}{n h^d_{2,n}}\sum_{i=1}^{n}L\left(\frac{t-X_i}{h_{2,n}}\alpha(f(X_i))\right)
(\alpha^d)'( f(X_i)) D(X_i;h_{1,n})\label{D}\\
&=\frac{1}{n^2h^d_{1,n}h^d_{2,n}}\sum_{i=1}^n (H_t(X_i,X_i)-\mathbb{E}_XH_t(X_i,X))\label{ct4}\\
&+\frac{n-1}{nh^d_{1,n}h^d_{2,n}}U_n\left(\pi_2(H_t(\cdot,\cdot))\right)\label{ct5}\\
&+\frac{n-1}{n^2h^d_{1,n}h^d_{2,n}}\sum_{i=1}^n(\mathbb{E}_XH_t(X, X_i)-\mathbb{E}H_t)\label{ct6}.
\end{align}
Obviously, (\ref{ct5}) and (\ref{ct6}) have mean zero since 
$$\mathbb{E}U_n\left(\pi_2(H_t(\cdot,\cdot))\right)=\mathbb{E}(\mathbb{E}_XH_t(X, Y)-\mathbb{E}H_t)=0.$$
In the empirical process (\ref{ct4}), set
$\bar Q_i(t)=H_t(X_i,X_i)-\mathbb{E}_YH_t(X_i,Y)$
and observe that, 
$$\mathbb{E}|\bar Q_1(t)|\le B h_{2,n}^d,$$
for some finite constant $B$.   By combining the above analysis and by the analysis of Gin\'{e} and Sang  \cite{GineSang1} on page 144, we have 
\begin{align}
(\ref{eps1})&=\frac{1}{d}\mathbb{E}((\ref{D}))
=\frac{1}{d}\mathbb{E}\left(\frac{1}{n^2h^d_{1,n}h^d_{2,n}}\sum_{i=1}^n (H_t(X_i,X_i)-\mathbb{E}_YH_t(X_i,Y))\right)\notag\\ 
&=O\left(\frac{1}{nh^d_{1,n}}\right).\label{ustats}
\end{align}
By the analysis in (\ref{idealmean})-(\ref{biasbias}) and (\ref{ustats}),  the bias 
\begin{eqnarray*}
\mathbb{E}(\hat f(t;h_{1,n},  h_{2,n}))-f(t)
=\left(\sum_{|v|=4}\tau_vD_v(1/f)/v!\right)h_{2,n}^4 +o(h_{2,n}^4).
\end{eqnarray*}
\end{proof}

\section{Variance of the ideal estimator}\label{variance}

 We develop the second moment expansion uniformly to deal with the variance of the ideal estimator.  Here we denote $h=h_{2,n}$ and $\gamma(s)=\alpha(f(s))$ for convenience. Then, the ideal estimator (\ref{ideal}) has the form  
\begin{equation*} 
\bar f(t;h)=\frac{1}{nh^d}\sum_{i=1}^n \gamma^d(X_i)K(h^{-1}\gamma (X_i)(t-X_i)),\ \ t\in\mathbb R^d.
\end{equation*}
Denote $A(X_i) = \gamma^d(X_i)K(h^{-1}\gamma(X_i)(t-X_i))$, then we have the second moment of the ideal estimator as follows:
\begin{eqnarray*}
\mathbb{E}\bar{f}^2(t;h)&=&\frac{1}{n^2h^{2d}}\mathbb{E}\left(\sum_{i=1}^nA(X_i)\right)^2 \\
&=&\frac{1}{n^2h^{2d}}\sum_{i=1}^n\mathbb{E}A^2(X_i)+\frac{1}{n^2h^{2d}}\sum_{i\neq j}\mathbb{E}A(X_i)\mathbb{E}A(X_j)\\
&=&\frac{1}{nh^{2d}}\mathbb{E}A^2(X_1)+\frac{n(n-1)}{n^2h^{2d}}(\mathbb{E}A(X_1))^2.
\end{eqnarray*}
Recall that
$$\mathbb{E}\bar{f}(t;h)=f(t)+\left(\sum_{|v|=4}\tau_vD_v(1/f)/v!\right)h^4 +o(h^4)$$
if $f(t)>t_0c^2$. Since $\mathbb{E}A(X_1)=h^d\mathbb{E}\bar{f}(t;h)$, we have  
\begin{eqnarray}\label{var1}
Var\bar{f}(t;h)&=&\mathbb{E}\bar{f}^2(t;h)-[\mathbb{E}\bar{f}(t;h)]^2\notag\\
&=&\frac{1}{nh^{2d}}\mathbb{E}A^2(X_1)+\frac{(n-1)}{nh^{2d}}(\mathbb{E}A(X_1))^2-\frac{1}{h^{2d}}(\mathbb{E}A(X_1))^2\notag\\
&=&\frac{1}{nh^{2d}}\mathbb{E}A^2(X_1)-\frac{1}{n}\Bigg[f(t)+\left(\sum_{|v|=4}\tau_vD_v(1/f)/v!\right)h^4 +o(h^4)\Bigg]^2\notag\\
&=&\frac{1}{nh^{2d}}\mathbb{E}A^2(X_1)+O(n^{-1}).
\end{eqnarray}
In the next proposition, we study the quantity $\mathbb{E}A^2(X_1)$. The idea is similar to the uniform bias expansion as in McKay \cite{McKay b}, Jones,  McKay and Hu \cite{JonesMcKayHu}, and particularly Gin\a'{e} and Sang \cite{GineSang1}. 
\begin{proposition} \label{bias0} 
Suppose that the kernel $K$ on $\mathbb{R}^d$ has the form $K(t)=\Phi(\|t\|^2)$  for some real function $\Phi$  with uniformly bounded second order derivative and with support contained in $[0,T]$, $T<\infty$. $K$ is non-negative and integrate to 1. Assume the density function $f$ is in $C^l(\mathbb R^d)$. Suppose that $\gamma(t)\ge c>0$ for some $c>0$ and all $t\in \mathbb R^d$, and that the function $\gamma(t)$ is in $C^{l+1}(\mathbb R^d)$. Then we have,
\begin{equation}\label{locationideal3}
\mathbb{E}A^2(X_1)=\sum_{k=0}^l a_{k}(t)h^{k+d}+o(h^{l+d})
\end{equation}
as $h\to 0$, uniformly in $t\in\mathbb R^d$. The set of functions $a_{k}$, which are uniformly bounded and equicontinuous, are defined as 
\begin{equation}\label{coeff}
a_{2k+1}(t)=0,\ \ a_{2k}(t)=\sum_{|v|=2k}\frac{\mu_{v}}{v!}D_v \frac{f(t)}{\gamma^{2k-d}(t)},
\end{equation}
uufor $k\le l/2$, in particular, $a_{0}(t)=\gamma^d(t)f(t)\mu_0$. Here $|v|, v!, \mu_v$ and $D_v$ are defined in (\ref{notation}).
\end{proposition}

\begin{proof}
Note that there exists $\delta_1>0$ such that $a_{ii}=\gamma(t-v)+v_i\frac{\partial\gamma(t-v)}{\partial v_i}$, $1\le i\le d$, are bounded away from zero, $a_{ij}=v_i\frac{\partial\gamma(t-v)}{\partial v_j}$, $1\le i\le d, 1\le j\le d, j\ne i$, are small enough for all $t\in \mathbb R^d$ and $v\in [-\delta_1,\delta_1]^d$ for functions $\gamma$ that are bounded away from zero and their derivatives that are bounded. Hence the matrix $A=(a_{ij})_{i,j=1}^d$ is invertible. Thus the vector function $v\mapsto U_{t}(v):=v\gamma(t-v)$ is invertible on the neighborhood $[-\delta_1,\delta_1]^d$ of $v=0$ for each $t\in\mathbb R^d$.  By differentiation, it is easy to see that the inverse function, say $V_{t}(u)$,  is $l+1$ times differentiable with continuous partial derivatives. Unless $\|t-s\|^2\le h^2T/c^2$, $K(h^{-1}\gamma(s)(t-s))=0$. Hence, the change of variables
$$hz=(t-s)\gamma(t-(t-s)),\ {\rm that\ is,}\ t-s=V_{t}(hz),$$ in the following integral is valid for all $h$ small enough
\begin{eqnarray}
\mathbb{E}A^2(X_1)&=&\int\gamma^{2d}(s)f(s)K^2\left(\frac{t-s}{h}\gamma(s)\right)ds\label{2moment}\\
&=&-h^d\int \gamma^{2d}(t-V_{t}(hz))f(t-V_{t}(hz))|\det(J)|K^2(z)dz\notag
\end{eqnarray}
where $J$ is the partial derivative matrix of the vector function $V_{t}(hz)$ with respect to $hz$ and $\det(J)$ is the determinant of $J$.   If we develop the function $\gamma^{2d}(t-V_{t}(hz))f(t-V_{t}(hz))|\det(J)|$ into powers of $hz$ and then integrate it, noting the compactness of the domain of integration and the differentiability properties of $f$ and $\gamma$, we have $(\ref{locationideal3})$.

 Suppose  $\psi$ is infinitely differentiable and has bounded support. Then, changing variables ($t=s+hu$) from $t$ to $u$ in (\ref{2moment}), developing $\psi$, changing variables once more ($w=u\gamma(s)$)) and  integrating by parts, we obtain
 \begin{eqnarray}\label{expansion2}
 &&\int \psi(t)\mathbb{E}A^2(X_1) dt =h^d\int\int \psi(s+hu)\gamma^{2d}(s)f(s)K^2(u\gamma(s))duds\\
 &=&h^d\int\gamma^{2d}(s)f(s)\int \psi(s+hu)K^2(u\gamma(s))duds\notag\\
 &=&h^d\int\gamma^{2d}(s)f(s)\int\sum_{|v|=0}^l\frac {D_v\psi(s)}{k!}h^{|v|} u^v K^2(u\gamma(s))duds+o(h^{l+d})\notag\\
 &=&h^d\int\gamma^d(s)f(s)\int\sum_{|v|=0}^l\frac {D_v\psi(s)}{v!}\frac{h^{|v|}w^v}{\gamma^{|v|}(s)}K^2(w)dwds+o(h^{l+d})\notag\\
  &=&\sum_{|v|=0}^l\mu_vh^{|v|+d}\int f(s)\frac {D_v\psi(s)}{v!\gamma^{|v|-d}(s)}ds+o(h^{l+d})\notag\\
&=&\sum_{|v|=0}^l(-1)^{|v|}\mu_vh^{|v|+d}v!^{-1} \int  \psi(s)D_v(f(s)\gamma^{d-|v|}(s))ds+o(h^{l+d}).  \notag
 \end{eqnarray}
Here $u^v$ is defined in (\ref{notation}). 
Notice that $\mu_{2k+1}=0$ for $k\ge 0$. Then, (\ref{coeff}) 
follows by comparing the coefficients of $h^k$ in both expansions (\ref{locationideal3}) and (\ref{expansion2}).
 \end{proof}
 
Thus, the variance of the ideal estimator is $\frac{\gamma^d(t)f(t)\mu_0}{nh^d}(1+o(1))=\frac{\alpha^d(f(t))f(t)\mu_0}{nh^d}(1+o(1))$ by applying Proposition \ref{bias0}  and (\ref{var1}).

\section{Central limit theorem}\label{central}
\subsection{Central limit theorem for ideal estimator}
The ideal estimator $\bar{f}(t;h_{2,n})$ in (\ref{ideal}) can be written as a sample mean of triangular array of i.i.d. random variables, i.e., $\bar{f}(t;h_{2,n})=\bar{Y}=\frac{1}{n}\sum_{i=1}^n Y_{n,i}$, where 
\begin{equation}\label{Yi}
Y_{n,i} = \frac{1}{h^d_{2,n}}K(h_{2,n}^{-1}\alpha(f (X_i))(t-X_i))\alpha^d(f (X_i)).
\end{equation}
Notice that $\mathbb{E}Y_{n,i}^2 <\infty$. By (\ref{var1}) and Proposition \ref{bias0}, 
\begin{equation*}
\sqrt{Var(\bar{f}(t;h_{2,n}))}= (1+o(1))\sqrt{\frac{1}{nh^d_{2,n}}\alpha^d(f(t))f(t)\mu_0} .
\end{equation*}
 Hence, by the Lindeberg's central limit theorem for triangular array of random variables, we have the following central limit theorem for the ideal estimator for all $t\in \mathbb{R}^d$, 
\begin{equation*}
\sqrt{nh^d_{2,n}}[\bar{f}(t;h_{2,n})-\mathbb{E}\bar{f}(t;h_{2,n})] \xrightarrow{D} N(0,\alpha^d(f(t))f(t)\mu_0).
\end{equation*}
Since $\mathbb{E}\bar{f}(t;h_{2,n})-f(t)=\left(\sum_{|v|=4}\tau_vD_v(1/f)/v!\right)h_{2,n}^4(1+o(1))$ for $t\in {\cal D}_r$ by McKay (\cite{McKay a}, \cite{McKay b}) or Corollary 1 in Gin\'{e} and Sang \cite{GineSang1}, 
\begin{equation*}
\sqrt{nh^d_{2,n}}[\mathbb{E}\bar{f}(t;h_{2,n})-f(t)]=c_2^{(d+8)/2}\sum_{|v|=4}\tau_vD_v(1/f)/v!(1+o(1)),
\end{equation*}
if we take $h_{2,n}=c_2n^{-1/(8+d)}$ for some constant $c_2>0$. 
Note that, 
\begin{eqnarray*}
\bar{f}(t;h_{2,n})-f(t)=\bar{f}(t;h_{2,n})-\mathbb{E}\bar{f}(t;h_{2,n})+\mathbb{E}\bar{f}(t;h_{2,n})-f(t).
\end{eqnarray*}
Thus, by Slutsky's theorem, for $t\in {\cal D}_r$, 
\begin{equation*}
\sqrt{nh^d_{2,n}}[\bar{f}(t;h_{2,n})-f(t)] \xrightarrow{D} N\left(c_2^{(d+8)/2}\sum_{|v|=4}\tau_vD_v(1/f)/v!,\alpha^d(f(t))f(t)\mu_0\right).
\end{equation*}

\subsection{Central limit theorem for the true estimator}
Based on the above central limit theorems for the ideal estimator, we have the following central limit theorems for the true variable bandwidth kernel density estimator.  
\begin{theorem}\label{tclt}
Let {$X_1, ..., X_n$} be a random sample of size n with density function $f(t)$, $t \in \mathbb{R}^d$, and $\hat{f}(t; h_{1,n}, h_{2,n})$ defined as in  $(\ref{true})$ is an estimator of $f(t)$. Assume $f(t)$ to be in ${\cal P}_{C,4}$. Suppose the kernel $K$ on $\mathbb{R}^d$ has the form $K(t)=\Phi(\|t\|^2)$  for some real function $\Phi$  with uniformly bounded second order derivative and with support contained in $[0,T]$, $T<\infty$. $K$ is non-negative and  integrates to 1.  The function $\alpha (x)$ in the estimator $\hat{f}(t; h_{1,n}, h_{2,n})$ is defined in (\ref{alp}) for a nondecreasing clipping function $p(s)$ $[p(s)\ge 1$ for all $s$ and $p(s)=s$ for all $s\ge c\ge 1]$ with five bounded and uniformly continuous derivatives, and constant $c>0$.  Let $h_{2,n}= c_2n^{-1/(8+d)}$ for some constants $c_2>0$ and assume that $U(h_{1,n})=o(h_{2,n}^2)$. Then for $t\in {\cal D}_r$, 
\begin{equation}\label{clt1}
\sqrt{nh^d_{2,n}}[\hat{f}(t;h_{1,n},h_{2,n})-\mathbb{E}\hat{f}(t;h_{1,n},h_{2,n})]\xrightarrow{D} N\left(0,\sigma_t^2\right)
\end{equation}
and 
\begin{equation}\label{clt2}
\sqrt{nh^d_{2,n}}[\hat{f}(t;h_{1,n},h_{2,n})-f(t)]\xrightarrow{D} N\left(c_2^{(d+8)/2}\sum_{|v|=4}\tau_vD_v(1/f)/v!,\sigma_t^2\right).
\end{equation}
Here, $\sigma_t^2=\alpha^d(f(t))f(t)\mu_0+f^3(t)\frac{[(\alpha^{d})^{\prime}(f(t))]^2}{d^2\alpha^d(f(t))}\int_{\mathbb{R}^d} L^2(z)dz+f^2(t)(\alpha^d)^{\prime}(f(t))\mu_0$, $\mu_0=\int_{\mathbb{R}^d} K^2(u)du$, and  $L(x)=K(x)+xK'(x)$. 
\end{theorem}
\begin{proof}
\indent 
The true estimator $\hat{f}(t; h_{1,n}, h_{2,n})$ in  (\ref{true})
 has decomposition 
\begin{align}
\hat{f}(t;h_{1,n},h_{2,n})-\mathbb{E}\hat{f}(t;h_{1,n},h_{2,n})&=\hat{f}(t;h_{1,n},h_{2,n})-\bar{f}(t;h_{2,n})\label{dec1}\\
&+\bar{f}(t;h_{2,n})-\mathbb{E}\bar{f}(t;h_{2,n}) \label{dec2}\\
&+\mathbb{E}\bar{f}(t;h_{2,n})-\mathbb{E}\hat{f}(t;h_{1,n},h_{2,n}),\label{dec4}
\end{align}
\begin{align}
\hat{f}(t;h_{1,n},h_{2,n})-f(t)&=\hat{f}(t;h_{1,n},h_{2,n})-\mathbb{E}\hat{f}(t;h_{1,n},h_{2,n})\notag\\
&+\mathbb{E}\hat{f}(t;h_{1,n},h_{2,n})-f(t).\label{dec3}
\end{align}
Since 
\begin{equation*}
\sqrt{nh^d_{2,n}}[\mathbb{E}\bar{f}(t;h_{2,n})-\mathbb{E}\hat{f}(t;h_{1,n}h_{2,n})]=o(h_{2,n}\sqrt{nh^d_{2,n}})=o(1) 
\end{equation*}
by the analysis in Section \ref{BIAS}, the term (\ref{dec4}) is negligible in the central limit theorems (\ref{clt1})  and (\ref{clt2}). 
The term (\ref{dec1}) has decomposition as in (\ref{e'}) - (\ref{e3}). 
We know that $(\ref{e1})= O_{a.s.}\left(U^2(h_{1,n})\right)=o_{a.s.}(h_{2,n}^4)$, $(\ref{e2})=O_{a.s.}\left(U^3(h_{1,n})\right)=o_{a.s.}(h_{2,n}^6)$ and $(\ref{e3}) =O_{a.s.}\left(U^4(h_{1,n})\right)=o_{a.s.}(h_{2,n}^8)$ by (3.20) of Gin\'{e} and Sang  \cite{GineSang1}. Hence they are also negligible in the central limit theorems (\ref{clt1})  and (\ref{clt2}).

We can further decompose $(\ref{e'})$ into the random variation part $D$ and the bias $b$ by the decomposition (\ref{deltatalor}):
\begin{align}
(\ref{e'})&=\frac{1}{dnh^d_{2,n}}\sum_{i=1}^n\bigg[L\left(\frac{t-X_i}{h_{2,n}}\alpha(f(X_i))\right)
(\alpha^d)^{\prime}( f(X_i)) D(X_i;h_{1,n})\bigg]\label{ct1}\\
&+\frac{1}{dnh^d_{2,n}}\sum_{i=1}^n\bigg[L\left(\frac{t-X_i}{h_{2,n}}\alpha(f(X_i))\right)
(\alpha^d)^{\prime}(f(X_i))  b(X_i;h_{1,n})\label{ct2}\bigg]\\
&+\frac{1}{2nh^d_{2,n}}\sum_{i=1}^n\bigg[L\left(\frac{t-X_i}{h_{2,n}}\alpha(f(X_i))\right)\alpha^{d-1}(f(X_i))\alpha''(\eta(X_i))[\hat f(X_i;h_{1,n})-f(X_i)]^2\bigg].\label{ct3}
\end{align}
 By (3.24) of Gin\'{e}  and Sang  \cite{GineSang1}, we have that $(\ref{ct3}) = O_{a.s.}\left(U^2(h_{1,n})\right)=o_{a.s.}(h_{2,n}^4)$ and by the analysis for (3.27) of the same paper, we have $(\ref{ct2}) =o_{a.s.}(h_{2,n}^4)$. Also, the term $(\ref{ct1})$ multiplied by $d$ can be further decomposed into (\ref{ct4}), (\ref{ct5}) and (\ref{ct6}) from Section 2, and by the proof of (3.33) of Gin\'{e} and Sang  \cite{GineSang1},  $(\ref{ct5})=o_{a.s.}(h_{2,n}^4)$. Next we show that $(\ref{ct4})$=$o_{p}(h_{2,n}^4)$.
 
 \indent It is easy to see that $\mathbb{E}[H_t(X_1,X_1)-\mathbb{E}_YH_t(X_1,Y)]$ and $\mathbb{E}[H_t(X_1,X_1)-\mathbb{E}_YH_t(X_1,Y)]^2$ are bounded by $Ch^d_{2,n}$, for some constant $C$. Thus, 
\begin{align*}
&\mathbb{E}(\ref{ct4})^2:=\mathbb{E}B_n^2\\
&= \frac{1}{n^3h_{1,n}^{2d}h_{2,n}^{2d}}\mathbb{E}[H_t(X_1,X_1)-\mathbb{E}_YH_t(X_1,Y)]^2+\frac{n-1}{n^3h_{1,n}^{2d}h_{2,n}^{2d}}[\mathbb{E}(H_t(X_1,X_1)-\mathbb{E}_YH_t(X_1,Y))]^2\\
&\le \frac{C}{n^2h_{1,n}^{2d}}, \text{for some constant C.}
\end{align*}
Let $\epsilon >0$ be given. Then, by Markov's inequality, we have\\
\begin{align*}
\mathbb{P}\left(\left|\frac{B_n}{h_{2,n}^4}\right|>\epsilon \right) \le \frac{\mathbb{E}(B_n^2)}{\epsilon^2h_{2,n}^8}\le \frac{C}{n^2h_{1,n}^{2d}\epsilon^2h_{2,n}^8}\xrightarrow{n\rightarrow\infty}0.
\end{align*} 
Hence, $(\ref{ct4})$=$o_{p}(h_{2,n}^4)$.

By the above analysis, only the term (\ref{dec2}) and the remaining term from (\ref{dec1}), i.e.,  (\ref{ct6}) divided by $d$,  have contribution in the central limit theorems (\ref{clt1}) and (\ref{clt2}). The other terms are all negligible. Now let  $Z_{n,i}=\frac{1}{dh_{1,n}^d h_{2,n}^d}[\mathbb{E}_XH_t(X,X_i)-\mathbb{E}H_t]$, $1\le i\le n$. Define $R_{n,i}= Y_{n,i}+Z_{n,i}$, and $\bar{R}=\frac{1}{n}\sum_{i=1}^nR_{n,i}$ where $Y_{n,i}$ is defined as in (\ref{Yi}). To prove the central limit theorem (\ref{clt1}), it suffices to derive a central limit theorem for $\bar{R}-\mathbb{E}\bar{f}(t;h_{2,n})$ where $\bar{R}$ is the sample mean of i.i.d. random variables $R_{n,i}$, $1\le i\le n$.

We have $$\mathbb{E}\bar{R}=\mathbb{E}R_{n,1}=\mathbb{E}\bar{f}(t;h_{2,n})=f(t)+h_{2,n}^4 \sum_{|v|=4}\tau_vD_v(1/f)/v!+o(h_{2,n}^4)$$  by Corollary 1 of Gin\'{e} and Sang  \cite{GineSang1} and since  $\mathbb{E}Z_{n,1}=0$.
Also, $ h^d_{2,n}\mathbb{E}Y_{n,1}^2=\alpha^d(f(t))f(t)\mu_0+O(h^2_{2,n})$ by Proposition \ref{bias0}  in Section \ref{variance}.
Since 
\begin{align*}
&h^d_{2,n}Var(R_{n,1})=h^d_{2,n}[\mathbb{E}R_{n,1}^2-(\mathbb{E}R_{n,1})^2]\\
&=h^d_{2,n}\mathbb{E}Y_{n,1}^2+h^d_{2,n}\mathbb{E}Z_{n,1}^2+2h^d_{2,n}\mathbb{E}(Y_{n,1}Z_{n,1})-h^d_{2,n}(\mathbb{E}R_{n,1})^2,
\end{align*}
we need to calculate the limit of terms $h^d_{2,n}\mathbb{E}(Y_{n,1}Z_{n,1})$ and $h^d_{2,n}\mathbb{E}Z_{n,1}^2$. 

Let $x=uh_{1,n}+x_1$ and $x_1 = t-vh_{2,n}$ be the change of variables. Then,
\begin{align*}
&\frac{1}{h_{1,n}^{2d} h^d_{2,n}}\mathbb{E}(\mathbb{E}_XH_t(X,X_1))^2\nonumber\\
&=\frac{1}{h_{1,n}^{2d} h^d_{2,n}}\int_{\mathbb{R}^d} \left[\int_{\mathbb{R}^d} L\left(\frac{t-x}{h_{2,n}}\alpha(f(x))\right)(\alpha^d)^{\prime}(f(x))K\left(\frac{x-x_1}{h_{1,n}}\right)f(x) dx\right]^2 f(x_1)dx_1\\
&=\frac{1}{h^d_{2,n}}\int_{\mathbb{R}^d}\Big[\int_{\mathbb{R}^d} L\left(\frac{t-x_1-h_{1,n}u}{h_{2,n}}\alpha(f(x_1+h_{1,n}u))\right)(\alpha^d)^{\prime}(f(x_1+h_{1,n}u))K(u)\nonumber\\
&\times f(x_1+h_{1,n}u) du\Big]^2 f(x_1)dx_1\\
&=\int_{\mathbb{R}^d}\Big[\int_{\mathbb{R}^d} L\left(\left(v-\frac{h_{1,n}u}{h_{2,n}}\right)\alpha(f(t-h_{2,n}v+h_{1,n}u))\right)(\alpha^d)^{\prime}(f(t-h_{2,n}v+h_{1,n}u))\nonumber\\
&\times K(u)f(t-h_{2,n}v+h_{1,n}u) du\Big]^2 f(t-h_{2,n}v)dv\\
&\xrightarrow{n\rightarrow\infty}f^3(t)\frac{[(\alpha^{d})^{\prime}(f(t))]^2}{\alpha^d(f(t))}\int_{\mathbb{R}^d} L^2(z)dz,
\end{align*}
and 
\begin{align*}
&\frac{1}{h_{1,n}^d}\mathbb{E}[Y_{n,1}\mathbb{E}_X(H_t(X,X_1))]\\
&=\frac{1}{h^d_{1,n}h^d_{2,n}}\int_{\mathbb{R}^d} \int_{\mathbb{R}^d} L\left(\frac{t-x}{h_{2,n}}\alpha(f(x))\right)(\alpha^d)^{\prime}(f(x))K\left(\frac{x-x_1}{h_{1,n}}\right)K\left(\frac{t-x_1}{h_{2,n}}\alpha(f(x_1))\right)\\
&\times (\alpha^d)(f(x_1))f(x)f(x_1)dxdx_1\\
&=\int_{\mathbb{R}^d} \int_{\mathbb{R}^d} L\left(\left(v-\frac{uh_{1,n}}{h_{2,n}}\right)\alpha(f(uh_{1,n}+t-vh_{2,n}))\right)(\alpha^d)^{\prime}(f(uh_{1,n}+t-vh_{2,n})) K(u)\\
&\times K(v\alpha(f(t-h_{2,n}v)))(\alpha^d)(f(t-h_{2,n}v))f(t-h_{2,n}v)f(uh_{1,n}+t-vh_{2,n})dudv\\
&\xrightarrow{n\rightarrow\infty} \frac{1}{2}df^2(t)(\alpha^d)^{\prime}(f(t))\mu_0
\end{align*}
since $\int_{\mathbb{R}^d} K(t)\sum_{i=1}^d t_iK_i'(t)dt=-d\mu_0/2$. By the change of variables, it is easy to see  that $\mathbb{E}(H_t)$ is bounded by $Ch_{1,n}^dh_{2,n}^d$ for some $C>0$ and then $\frac{1}{h_{1,n}^{2d}h^d_{2,n}}[\mathbb{E}(H_t)]^2\rightarrow 0$ as $n\rightarrow \infty$. Hence,
\begin{align*}
h^d_{2,n}\mathbb{E}Z_{n,1}^2&=\frac{1}{d^2h_{1,n}^{2d} h^d_{2,n}}\mathbb{E}(\mathbb{E}_XH_t(X,X_1))^2 - \frac{1}{d^2h_{1,n}^{2d}h^d_{2,n}}[\mathbb{E}(H_t)]^2\\
&\xrightarrow{n\rightarrow\infty}f^3(t)\frac{[(\alpha^{d})^{\prime}(f(t))]^2}{d^2\alpha^d(f(t))}\int_{\mathbb{R}^d} L^2(z)dz
\end{align*}
and 
\begin{align*}
h^d_{2,n}\mathbb{E}(Y_{n,1}Z_{n,1})&=\frac{1}{dh_{1,n}^d}\mathbb{E}[Y_{n,1}\mathbb{E}_X(H_t(X,X_1))] -\frac{1}{dh_{1,n}^d}\mathbb{E}Y_{n,1} \mathbb{E}H_t\\
&\xrightarrow{n\rightarrow\infty} \frac{1}{2}f^2(t)(\alpha^d)^{\prime}(f(t))\mu_0.
\end{align*}
Thus,
\begin{align*}
h^d_{2,n}Var(R_{n,1})&\xrightarrow{n\rightarrow\infty} \sigma_t^2.
\end{align*}
Hence, by central limit theorem for i.i.d. random variables, we have
\begin{align*}
\sqrt{nh^d_{2,n}}[\bar{R}-\mathbb{E}\bar{R}] \xrightarrow{D} N\left(0,\sigma_t^2\right)
\end{align*}
and by Slustsky's theorem, 
\begin{equation*}
\sqrt{nh^d_{2,n}}[\hat{f}(t;h_{1,n},h_{2,n})-\mathbb{E}\hat{f}(t;h_{1,n},h_{2,n})]\xrightarrow{D}N\left(0,\sigma_t^2\right).
\end{equation*}
Since the term $(\ref{dec3})= \mathbb{E}\hat{f}(t;h_{1,n},h_{2,n})- f(t)=\sum_{|v|=4}\tau_vD_v(1/f)/v!h_{2,n}^4(1+o(1))$ by Proposition  \ref{bias12},
\begin{equation*}
\sqrt{nh^d_{2,n}}[\mathbb{E}\hat{f}(t;h_{1,n},h_{2,n})-f(t)]= c_2^{(d+8)/2}\sum_{|v|=4}\tau_vD_v(1/f)/v!(1+o(1)).
\end{equation*}
Thus,
\begin{equation*}
\sqrt{nh^d_{2,n}}[\hat{f}(t;h_{1,n},h_{2,n})-f(t)]\xrightarrow{D}N\left(c_2^{(d+8)/2}\sum_{|v|=4}\tau_vD_v(1/f)/v!,\sigma_t^2\right).
\end{equation*}
\end{proof}

With the central limit theorem in Theorem \ref{tclt}, one can have better statistical inference on the density function value at a fixed point $t$. For example, with some fixed confidence level, the confidence interval for the true density function $(f(t))$ at the fixed point $t$ using the variable bandwidth kernel estimation is better (the length of the confidence interval is shorter) than the classical case since the bandwidth $h_{2,n}$ here has order of $n^{-1/(8+d)}$ instead of $n^{-1/(4+d)}$.

\section{Simulation}\label{sim}
\indent In this section we evaluate the performance of the variable bandwidth kernel density estimator (VKDE), (\ref{true}),  in one dimensional case. Instead of the true estimator (\ref{true}), Jones,  McKay and Hu \cite{JonesMcKayHu} did simulation study for the ideal estimator (\ref{ideal}) in one dimensional case.  First of all, we provide a result on the integrated mean squared error (IMSE) of the VKDE and therefore a formula of optimal bandwidth. 
\begin{theorem}
Under the conditions in Proposition \ref{bias12} and Theorem \ref{tclt}, the IMSE on ${\cal D}_r$ is 
\begin{align}\label{MSE}
R(h_{1,n},h_{2,n})|{\cal D}_r = h_{2,n}^8 \int_{{\cal D}_r}(\sum_{|v|=4}\tau_vD_v(1/f)/v!)^2dt+\frac{1}{nh_{2,n}^d}\int_{{\cal D}_r} \sigma_{t}^2dt+o(h_{2,n}^8),
\end{align}
Furthermore, the optimal bandwidth $h_{2,n}^*$ is given by
\begin{equation}\label{opt}
h_{2,n}^* = \left[\frac{n\int_{{\cal D}_r}(\sum_{|v|=4}\tau_vD_v(1/f)/v!)^2dt}{\int_{{\cal D}_r} \sigma_{t}^2dt}\right]^{-1/(8+d)}.
\end{equation}
\end{theorem}
\begin{proof}
From the analysis of Theorem \ref{tclt}, it is clear that $Var(\hat{f}(t;h_{1,n},h_{2,n})=\frac{(1+o(1))\sigma_t^2}{nh_{2,n}^d}$ for $t\in D_r$. Together with Proposition \ref{bias12}, we have  (\ref{MSE}). The optimal bandwidth (\ref{opt}), which minimize the IMSE,  is obvious from the IMSE formula  (\ref{MSE}).
\end{proof}

We compare the performance of VKDE and KDE by conducting one dimensional simulation study of t-distribution ($t_4(0,1)$), Cauchy(0,1) and Pareto(0,1). 
The sample size is $n=50,000$ for each simulation study. For all the simulations, we use KDE as in $(\ref{classical})$ with the normal kernel function.  We use the code {\it density()} in the programming software R and the default bandwidth chosen by  R in the estimation for $t_4(0,1)$. For Cauchy(0,1) or Pareto(0,1), the code {\it density()} in R can not provide a classical kernel density estimate. Instead, we make new code and select the bandwidth which optimizes the performance among a variety of bandwidths. For VKDE, we assume that  $h_{1,n}=n^{-1/5}$, $h_{2,n}=n^{-1/9}$, and use the Tricube kernel: 
\begin{equation*}
K(u) = \frac{70}{81}(1-|u|^3)^3 1_{|u| \le 1}
\end{equation*}
in either the pilot kernel density estimator or the true estimator (\ref{true}). The following five time differentiable clipping function $p$ with $t_0=2$ (Gin\'{e} and Sang  \cite{GineSang1}) is applied:
\begin{displaymath}
p(t)=\left\{\begin{array}{ll}
1+\frac{t^6}{64}\left(1-2(t-2)+\frac{9}{4}(t-2)^2-\frac{7}{4}(t-2)^3+\frac{7}{8}(t-2)^4\right) & \textrm{if $0\le t\le2$}\\
t&\textrm{if $t\ge 2$}\\
1&\textrm {if $t\le 0$}\end{array}
\right..
\end{displaymath}

\begin{figure}[ht]
\centering
\includegraphics[ height=2.40in, width=2.50in]{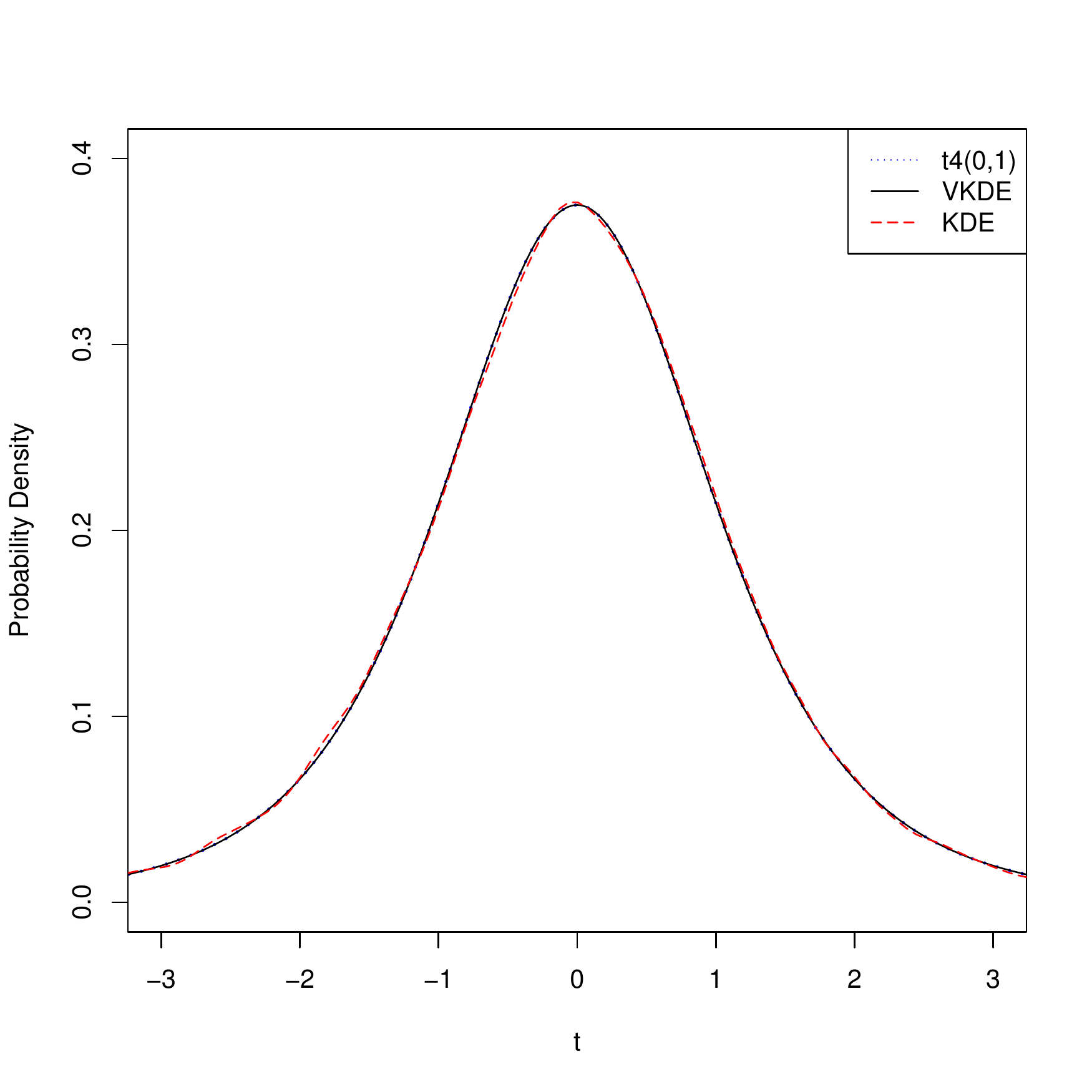}
\includegraphics[ height=2.40in, width=2.50in]{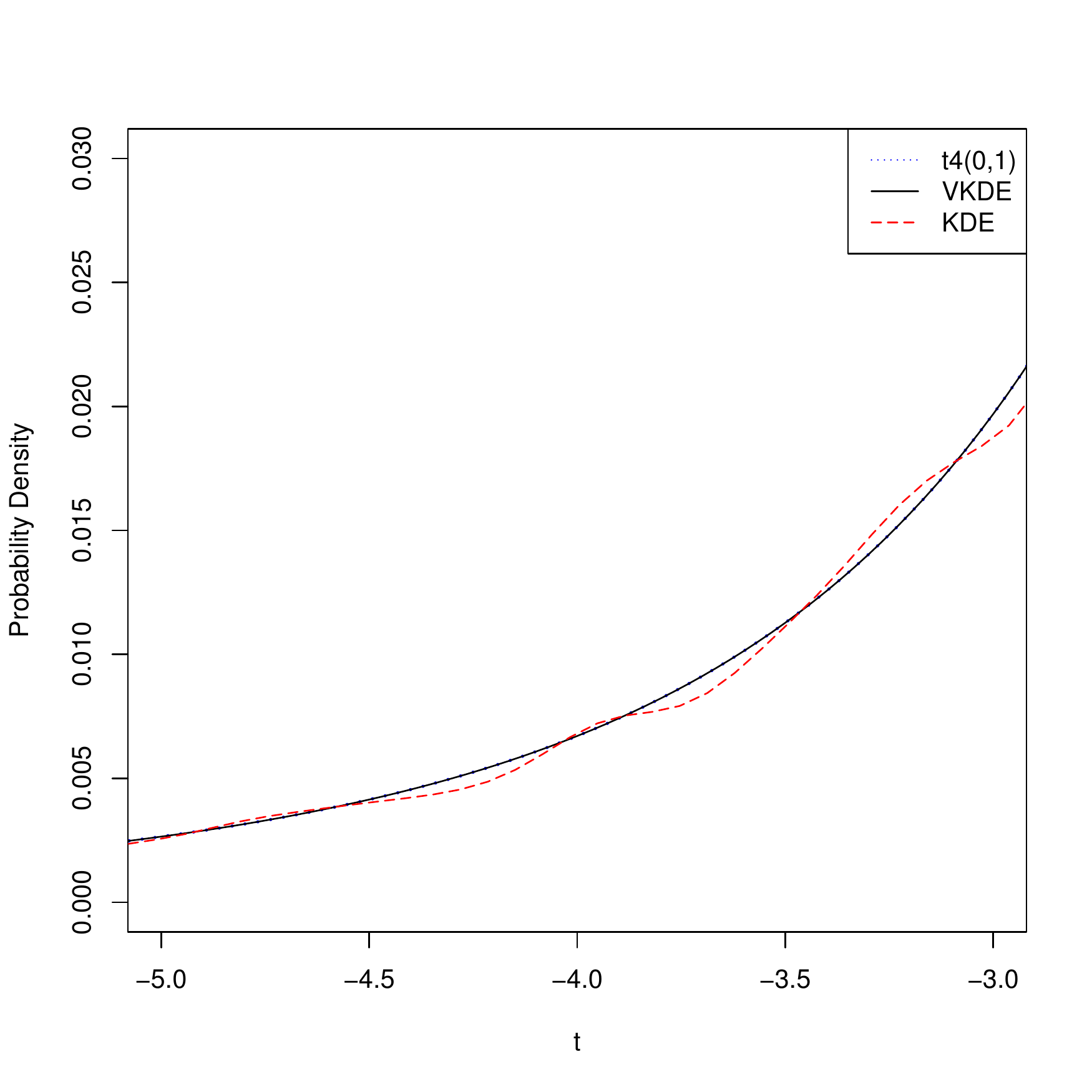}
\includegraphics[ height=2.40in, width=2.50in]{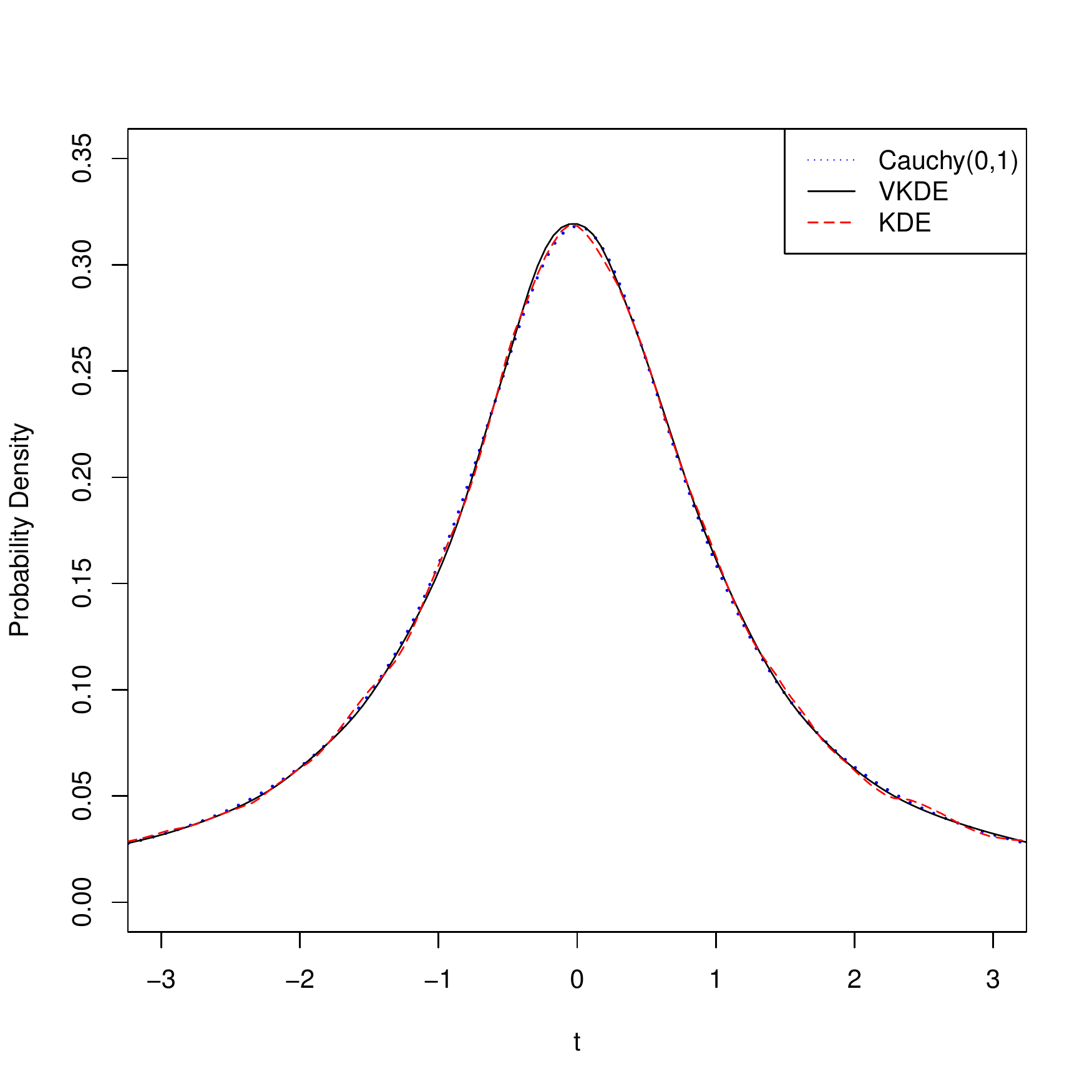}
\includegraphics[ height=2.40in, width=2.50in]{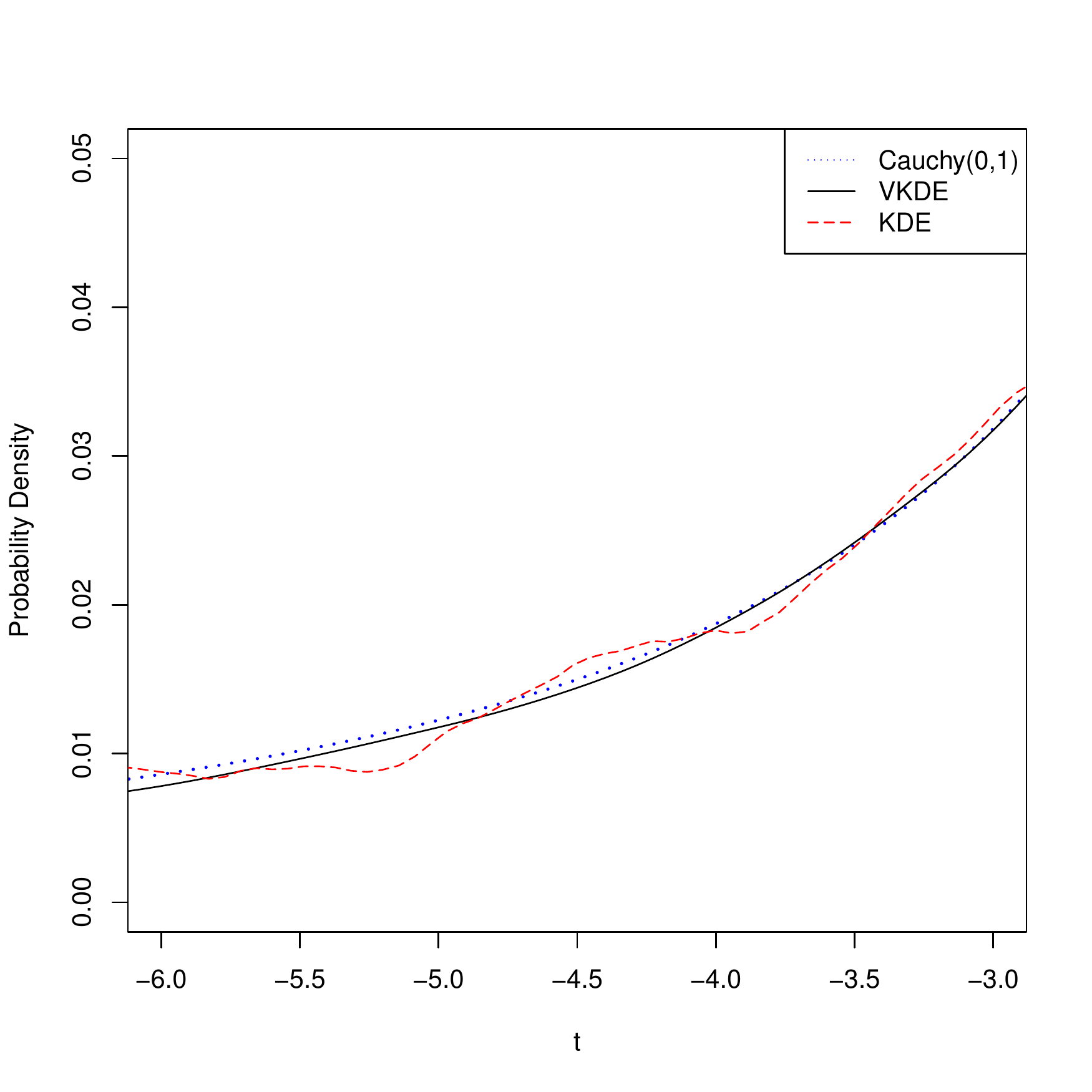}
\includegraphics[ height=2.40in, width=2.50in]{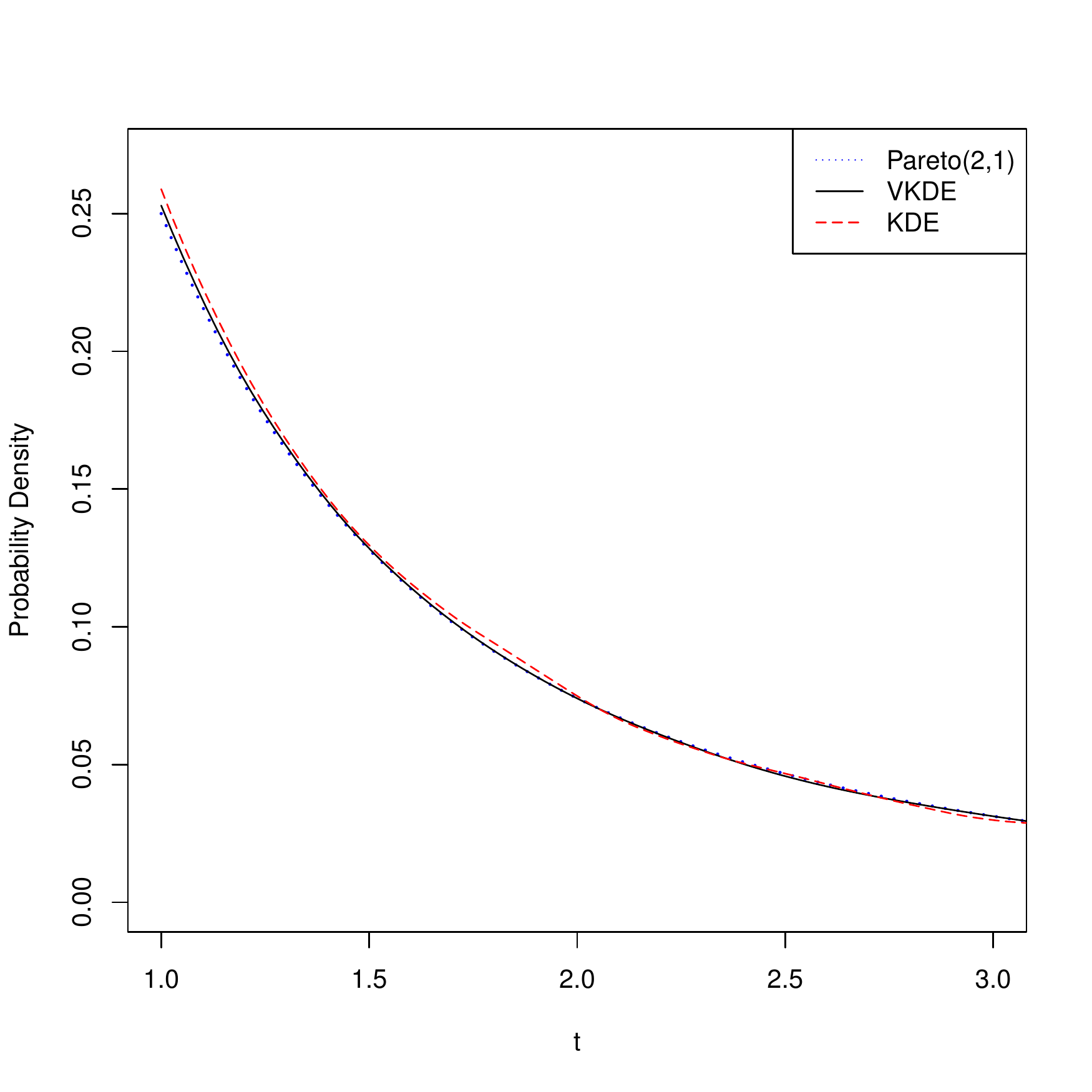}
\includegraphics[ height=2.40in, width=2.50in]{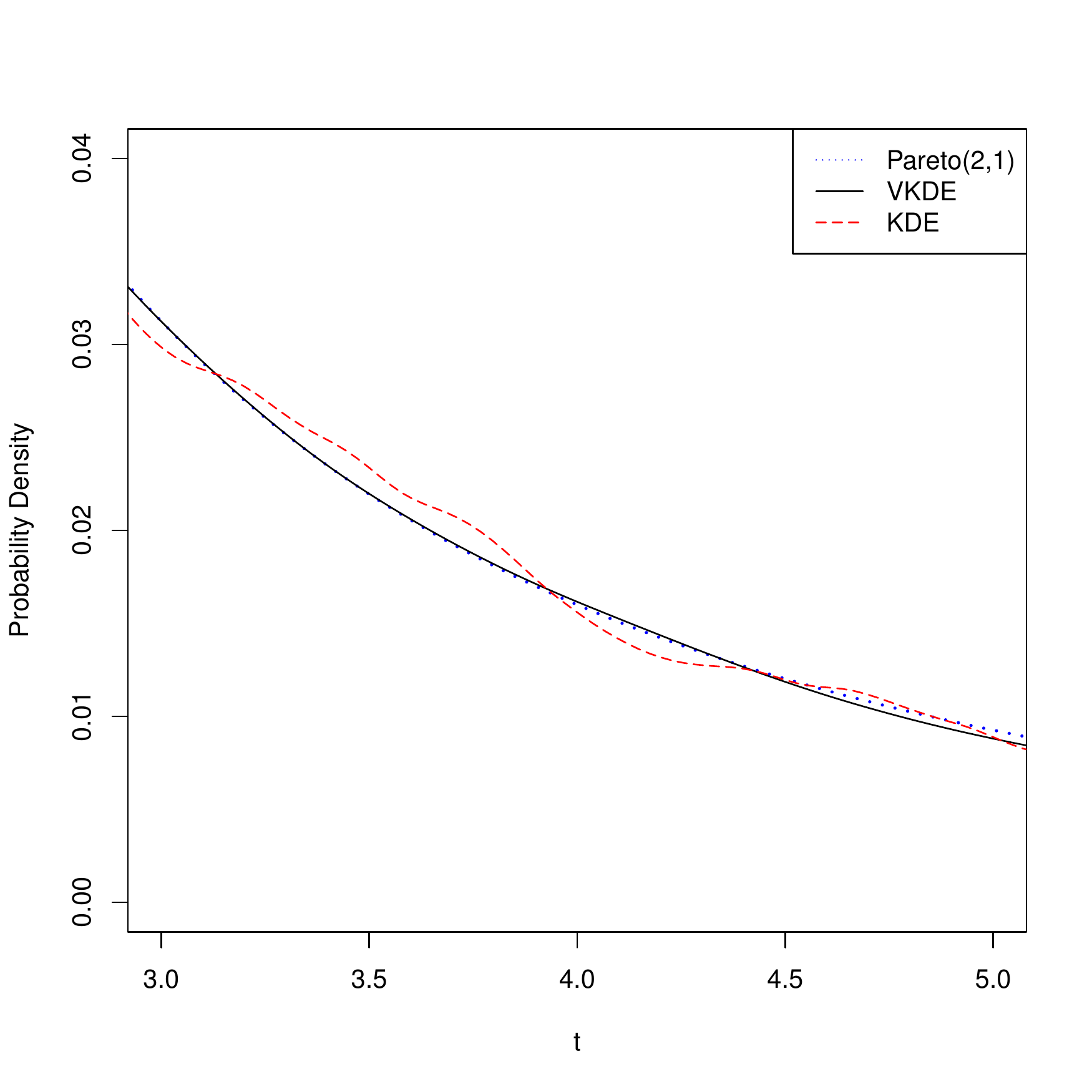}
\caption{The probability density functions of t-distribution ($t_4(0,1)$), Cauchy(0,1) and Pareto(0,1), the kernel density estimates (KDE),  and the variable kernel density estimates (VKDE) with 50,000 observations generated from t-distribution ($t_4(0,1)$), Cauchy(0,1) and Pareto(0,1) distribution. The left one shows the estimate in the main area with the mode. The right one shows the estimate in the tail area.}
\label{fig1}
\end{figure}
The simulation study in Figure \ref{fig1} shows that, for each of these three distributions,  VKDE has better performance than KDE, especially in the tail area. 

\vspace{1cm}
 
\noindent \textbf{Acknowledgement}  The authors thank the referee and the
Editor for their careful reading of the manuscript and for their insightful comments, 
which have helped to improve the quality of this paper.

\end{document}